\documentclass[10pt]{amsart}

\usepackage{amsfonts}
\usepackage{amsmath}
\usepackage{amssymb}
\usepackage{url}
\usepackage{setspace}
\usepackage{tikz}
\usepackage{mathtools}
\usepackage{booktabs}
\usepackage[colorlinks=true,linkcolor=blue,citecolor=blue,urlcolor=blue]{hyperref}

\newtheorem{theorem}{Theorem}
\newtheorem{lemma}[theorem]{Lemma}
\newtheorem{corollary}[theorem]{Corollary}
\newtheorem{proposition}[theorem]{Proposition}
\newtheorem{conjecture}[theorem]{Conjecture}

\theoremstyle{definition}
\newtheorem{definition}[theorem]{Definition}
\newtheorem{example}[theorem]{Example}
\newtheorem{remark}[theorem]{Remark}
\newtheorem{conclusion}[theorem]{Conclusion}
\newtheorem{observation}[theorem]{Observation}

\numberwithin{equation}{section}

\newcommand{\dd}{\delta}

\newcommand{\eps}{\varepsilon}
\newcommand{\N}{\mathbb{N}}
\newcommand{\Z}{\mathbb{Z}}
\newcommand{\R}{\mathbb{R}}
\newcommand{\PP}{\mathbb{P}}

\newcommand{\E}{\mathcal{E}}

\newcommand{\FF}{\mathcal{F}}
\newcommand{\TT}{\mathcal{T}}
\newcommand{\LL}{\mathcal{L}}
\newcommand{\supp}{\operatorname{supp}}
\newcommand{\LR}{\Leftrightarrow}
\begin{document}

\title[On the sum-of-digits measures and Cusick's conjecture via stopped random walks]{On the sum-of-digits measures and Cusick's conjecture via stopped random walks}

     \author[D.\ Tar\l{}owski]{Dawid Tar\l{}owski, ORCID 0000-0002-6824-4568; \\
      Faculty of Mathematics and Computer Science,
     Jagiellonian University,   \L{}ojasiewicza 6, 30-348 Krak\'{o}w, Poland}
  \thanks{E-mail addresses: \texttt{dawid.tarlowski@uj.edu.pl} ; \texttt{dawid.tarlowski@gmail.com}}

\begin{abstract}
Let $s(n)$ denote the number of ones  in the binary expansion of a natural 
number $n\in\mathbb{N}$. For any $t\in\mathbb{N}$ and $d\in\mathbb{Z}$, 
let $\mu_t(d)$ denote the asymptotic density of the set of those natural numbers $n$ for which $s(n+t)-s(n)=d$. 
The $\mu_t$ are properly defined probability measures on $\Z$, 
and the Cusick conjecture states that $\mu_t(\mathbb{N})>\frac{1}{2}$ for any 
$t\in\mathbb{N}$. We investigate the  properties of the family $\{\mu_t\}_{t\in\N}$
 by reindexing the odd integers via a suitable partial order. 
This construction leads to a nonautonomous dynamics on pairs of probability
 measures on $\Z$, which represents the process of growing a tree. The associated stopped random walk allows a transparent structural description of those measures, including their support, 
symmetries, variance, and an asymptotic dichotomy between the central limit theorem and the almost sure convergence. Next, we focus on  the median-preserving property of this process, and show that the Cusick conjecture is a special case of a more general claim about the asymmetric evolution of the associated binary trees, which we support numerically.

\end{abstract}

\maketitle

\section{Introduction}
Let $s(n)$ denote the number of ones in the binary expansion of an integer $n\in\N$, namely,
\[\textstyle s(n)=\sum_{k=0}^{m}n_k,\]
where: \begin{equation}\textstyle n=\sum_{k=0}^{m}n_k\cdot 2^k\mbox{ and } n_k\in\{0,1\}\mbox{ for }k=0,1,\dots,m. \end{equation}

 Throughout the paper we write $\N=\{0,1,2,\dots\}$ and $\N^+=\{1,2,3,\dots\}$. Given any $t\in\mathbb{N}$ and $d\in\mathbb{Z}$, let $\mu_t(d)$ denote the asymptotic density of the set of those natural numbers $n$ for which $s(n+t)-s(n)=d$, i.e.
\begin{equation}\label{AL}
\mu_t(d)=\lim_{N\to\infty} \frac{1}{N}|\{ n< N \colon s(n+t) -s(n)=d \} |.
\end{equation}
It is well known that the limit \eqref{AL} exists, and $\mu_t$ is a properly defined probability measure on $\mathbb{Z}$. Moreover, all the measures $\mu_t$ have mean zero, i.e. $\sum\limits_{k\in\mathbb{Z}}k\cdot\mu_t(k)=0$, and the following recurrence relations are satisfied:
\begin{equation}\label{even}
\mu_{2t}(d)=\mu_t(d),\ t\in\N,\ d\in\Z,
\end{equation}
\begin{equation}\label{odd}
\mu_{2t+1}(d)=\frac{1}{2}\mu_t(d-1)+\frac{1}{2}\mu_{t+1}(d+1),\ t\in\N,\ d\in\Z.
\end{equation}
For more details, see Lemma 1 in \cite{besineau}, Lemma 2.1 in \cite{drmota}. This paper is motivated by the Cusick's conjecture which states that
\begin{equation}\label{Cc}
{\textstyle{\mu_t(\N)>\frac12\mbox{ for any }t\in\N^+.}}
\end{equation}
The simplicity of the above statement is rather deceptive - the full conjecture is still open, although various  results on the subject have been established (\cite{drmota,Emme1,Emme,Hosten,sp2022,sw2019,sw2023}). We briefly highlight some of them: paper \cite{drmota} has shown that for any $\eps>0$ we have $|\{t<T\colon \frac12<\mu_t(\N)<\frac{1}{2}+\eps\}|=T-O(\frac{T}{\log T})$ (the symbol $O(\cdot)$ stands for big $O$ notation) which implies that the asymptotic density of the set of $t\in\N$ which satisfy $\mu_t(\N)>\frac12$ equals to one. In \cite{Emme}, a central limit theorem is established for generic sequences $(\mu_{t_k})$ arising by sampling digits of $t_k$ from the balanced Bernoulli measure. Papers \cite{Hosten} and \cite{sw2023} have generalized this result, and additionally \cite{sw2023} has shown that $\mu_t(\N)>\frac12$ for any $t$ such that $(t)_2$ contains sufficiently many blocks of ones. Recently, \cite{sobolewski} has proposed the decomposition of the characteristic functions associated to $\mu_t$ into the sum of the corresponding components. Finally, we  mention that Cusick conjecture has several equivalent formulations, see Section 3 in \cite{drmota}, and is  connected to Tu-Deng conjecture \cite{CusickLi},\cite{TuTuDeng},\cite{sw2019}.

A full proof of Cusick's conjecture is likely to require a detailed structural understanding of the family $\{\mu_t\}_{t \in \N}$. We will show that, after reindexing the odd integers via an appropriate partial order $\preceq$,\  the family $\{\mu_t\}_{t\in 2\N+1}$ corresponds to the marginal distributions of a hierarchical martingale defined by a stopped random walk, governed by an explicit recursive dynamics. For convenience we will assume that the random walk starts from zero, that is, we will work with the system of measures $P_t$ which is related to $\mu_{t}$ by the  convolution: $\mu_t=\mu_1\ast P_t$. First we will show that the maximal chains of the poset $(2\N+1,\preceq)$ determine the nonautonomous dynamics  on  pairs of probability measures on $\mathbb{Z}$ which represents the growth of a planar binary tree, and next we will express it in the language of the stopped random walk. The marginal probability distributions of the associated martingale - namely the measures $P_t$ - exhibit several rather desirable properties: the monotonicity of the span of the support, the monotonicity of the variance with the explicit bounds (Theorem \ref{WAR}), the symmetries, and clear asymptotic behaviour:  the unbounded variance case satisfies the central limit theorem (by \cite{Hosten},\cite{sw2023}) while in  the case of bounded variance the trajectories of the martingale are convergent, and the possible asymptotic distributions are exactly the measures $\mu_{t}$ and their reflections (Theorem \ref{As}). At the end of the paper we focus on the median-preserving property of the martingale, and show that the Cusick conjecture is a special case of a more general statement concerning the asymmetric growth of the trees.  In the language of binary digits: if $P_t(\N)>\frac12$ for any odd $t=(t_1t_2\dots t_n)_2$ with $t_2=0$, then $\mu_t(\N)>\frac12$ for any $t\in\N$. We conjecture  that the side of the tree that carries mass greater than $\frac12$ at the first step of the growth, remains so  throughout the whole process (Conjecture \ref{Azyan}). This last claim is supported numerically, and left open.


\section{  The  evolution of sum of digits measures.}\label{evo}
\subsection{Order}
Let us start with the sequence $x_n=s(n+1)-s(n)$, $n\in\N$, which defines measure $\mu_1$. Let $$\tau_k=\min\{n\in\N\colon x_n=k\}\mbox{, where }k\in\mathbb{Z}_{\leq 1}=\{d\in\mathbb{Z}\colon d\leq 1\}.$$
It is not difficult to establish that $x_n$ is given by:
 $$\tau_k=2^{1-k}\mbox{ and } x_{\tau_k+i}=x_i \mbox{ for } 1\leq i <2^{1-k},\mbox{ where } k\in\mathbb{Z}_{\leq 1}.$$ The first elements of $\{x_n\}_{n\in\N}$ are:\\
\small{
$$\textbf{1},\ \  \textbf{0}, 1, \ \ \textbf{-1}, 1, 0, 1,\ \ \textbf{-2}, 1, 0, 1, -1, 1, 0, 1,\ \ \textbf{-3},1, 0, 1, -1, 1, 0, 1,-2, 1, 0, 1, -1, 1, 0, 1,\textbf{-4},\dots\\ $$}

The above sequence  is an elegant representation of a geometric distribution with mean zero, more precisely, the corresponding frequencies lead to the formula:
\begin{equation}\label{mu1}
\mu_1(k)=\left(\frac{1}{2}\right)^{2-k}\mbox{ for } k=1,0,-1,-2,\dots
\end{equation}
 
In this chapter we want to gain good insight into how the measures $\mu_t$ evolve from $\mu_1$. By \eqref{even}, only odd $t$ have to be considered. 
To see that, we will consider an appropriate partial order on the set of odd natural numbers. 

We start from describing the intuition behind the above idea. Roughly speaking, equations \eqref{even} and \eqref{odd} lead to the following tree-structured decomposition of a natural number $t$ \textbf{:}  If $t$ is even, divide it by $2$. If $t$ is odd and greater than one, split $t$ into two numbers: $\frac{t-1}{2}$ and $\frac{t+1}{2}$.  For every number that has appeared, repeat this operation  until one is reached\ \textbf{.}  We will reverse time in this decomposition, and omit the even numbers, thereby obtaining a more straightforward evolution of $\mu_t$.
The initial measure $\mu_1$ is already defined.  By \eqref{even} and \eqref{odd}, the $\mu_3$ is given by: 
\begin{center}
$\mu_3(k)=\frac{1}{2}\mu_1(k+1)+\frac{1}{2}\mu_1(k-1),\ k\in\mathbb{Z}$.
\end{center} 
The further evolution will be analysed with use of  the following order.
\begin{definition}\label{order}
Define two maps $L, R \colon 2\N+1 \to 2\N+1$ by
\begin{equation}\label{eq:left-right}
L(t) = 2t - 1, \qquad R(t) = 2t + 1.
\end{equation}
We define a partial order $"\preceq"$ on $2\N+1$ as follows:
 $s \preceq t$ if and only if either $s=t$, or there exists a finite sequence $w_1, \ldots, w_n \in \{L, R\}$
such that
\begin{equation}\label{order-def}
t =w_n \circ \cdots \circ w_1(s).
\end{equation}
\end{definition}
 From now on we will write 
$$T=2\N+1,$$
and we consider this set as equipped with $\preceq$.

\begin{example}\label{ex:tree-levels}
The first four levels of $(T, \preceq)$ are
depicted below.
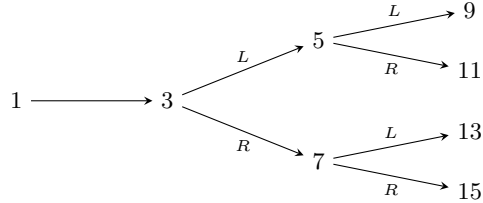
\begin{center}
\begin{tikzpicture}[
  grow=right,
  level distance=20mm,
  sibling distance=8mm,
  every node/.style={font=\small},
  edge from parent/.style={draw, -{stealth}, thin},
  level 1/.style={sibling distance=32mm},
  level 2/.style={sibling distance=16mm},
  level 3/.style={sibling distance=8mm},
]
\node {$1$}
  child {node {$3$}
    child {node {$7$}
      child {node {$15$} edge from parent node[below] {\tiny $R$}}
      child {node {$13$} edge from parent node[above] {\tiny $L$}}
      edge from parent node[below] {\tiny $R$}
    }
    child {node {$5$}
      child {node {$11$} edge from parent node[below] {\tiny $R$}}
      child {node {$9$} edge from parent node[above] {\tiny $L$}}
      edge from parent node[above] {\tiny $L$}
    }
    edge from parent node[above] {}
  };
\end{tikzpicture}
\end{center}
\end{example}

 There is a direct relation between the representation \eqref{order-def} and the binary representation of an odd integer. Let $\beta \colon \{L, R\} \to \{0, 1\}$ be a mapping defined by: 
$$\beta(L) = 0\mbox{ and }\beta(R) = 1.$$
\begin{observation}\label{binarr}
 The binary representation of $t=w_k \circ \cdots \circ w_1(3)$ is given by:
$$  t = (1\, \beta(w_1)\, \beta(w_2) \cdots \beta(w_k)\, 1)_2.$$
\end{observation}

To prove the above observation it is enough to show by induction that $t =w_k \circ \cdots \circ w_1(3)$ satisfies:
$$t = 2^{k+1} + \sum_{i=1}^{k} \beta(w_i)\,
2^{k+1-i} + 1.$$

Throughout the paper,  the set of all finite words $\{L,R\}^{\star}$ equipped with the standard prefix order will be identified with the set $T\setminus\{1\}$ with the order induced from  $T$. More precisely, define:
$$\{L,R\}^\star =\bigcup\limits_{n\in\N} \{L,R\}^n,\mbox{ where } \{L,R\}^0:=\{\eps\} \ \ \  (\eps - \mbox{an empty word }).$$
To the elements of $\{L,R\}^\star$ we will refer as words rather than sequences, and thus given $w\in\{L,R\}^\star$ determined by $w_1,\dots,w_n\in \{L,R\}$ we will write $w=w_1 \dots w_n$ rather than $w=(w_1,\dots,w_n)$. The length of the word $w$ will be denoted by $\ell(w)$.  Given $t\in T$, we will write: $tL:=L(t)$ and $tR:=R(t)$. More generally, given $w_1\dots w_n\in\{L,R\}^\star$, 
$$tw_1\dots w_n:=w_n\circ\dots\circ w_1 (t).$$
With the above notation, the following identification (a bijection) arises naturally:
$$h\colon \{L,R\}^\star\longrightarrow T\setminus\{1\},\mbox{ where }h(\varepsilon):=3\mbox{ and }$$
$$h\colon  w_1 \dots w_n\longmapsto 3w_1\dots w_n=w_n\circ\dots\circ w_1 (3).$$
We will write, for instance, $RR=15$ and $RRRLR=123$. With this notation, order $\preceq$ is consistent with the standard prefix order: $w_1w_2\dots 
w_n\preceq v_1v_2\dots v_m$ if and only if $n\leq m$ and $w_i=v_i$ for any $i\leq n$. \\

Given an infinite sequence $w=(w_1,w_2,\dots)\in\{L,R\}^{\N}$, we will write 
\begin{equation}\label{www}
w(-1):=1,\ w(0):=3,\ w(t):=w_1\dots w_t\in T,\ t\in\N^+.
\end{equation}
With the above notation, given any $w\in\{L,R\}^{\N}$, the sequence $\{w(t)\}_{t=-1}^{\infty}$ is the maximal chain in $(T,\preceq)$. We will analyze how
the sequences $w\in\{L,R\}^\N$ govern the evolution of:
$$\{\mu_{w(t)}\}_{t=-1}^{\infty}.$$

\subsection{The dynamics of $\mu_{w(t)}$}
The set of all probability measures on $\Z$ will be identified with the set $\Delta(\Z)\subset l^1(\Z)$ given by:
$$\Delta(\Z)=\{\ (P(k))_{k\in\Z}\ \colon\ P(k)\geq 0\mbox{ for all }k, \mbox{ and } \sum\limits_{k\in\Z}P(k)=1\}.$$
Set $\Delta(\Z)$ is naturally equipped with the topology of pointwise convergence which coincides with the topology induced from the space $(l^1(\Z),||\cdot||_1)$, where  $||P||_1=\sum\limits_{k\in\N}|P(k)|$. The set of centered probability measures on $\Z$ will be identified with the set:
$$\Delta_0(\Z)=\{\ P\in \Delta(\Z)\colon\  \sum\limits_{k\in\Z}k\cdot P(k)=0\}.$$
Above, by writing $\sum\limits_{k\in\Z}k\cdot P(k)=0$ we implicitly assume that the first moment exists: $\sum\limits_{k\in\Z}|k|\cdot P(k)<\infty$. Now, let $\{\sigma_d\}_{d\in\Z}$ denote the family of shift operators:
$$\sigma_d\colon \Delta(\Z)\ni (P(k))_{k\in\Z}\longmapsto  (P(k-d))_{k\in\Z}\in \Delta(\Z).$$
 We will write 
\begin{equation}\label{SLR}
\sigma_L:=\sigma_{-1}\mbox{ and }\sigma_R:=\sigma_{1}.
\end{equation}

\begin{definition}\label{Phi}
Define the operation $\Phi \colon \Delta_0(\Z) \times\Delta_0(\Z) \to\Delta_0(\Z)$, by
\[
  \Phi(\mu, \nu) = \frac{1}{2} \sigma_L(\mu) + \frac{1}{2} \sigma_R(\nu).
\]
\end{definition}
Any $w\in\{L,R\}^\N$  determines the sequence $\{\mu_{w(n)}\}_{n=-1}^{\infty}$ through the recursion based on the map $\Phi$.  For $t=2s+1\geq 3$, define:
\begin{equation}\label{dee}
\mu^L_{t}=\mu_{s+1}\mbox{ and }\mu^R_{t}=\mu_{s},
\end{equation}
so we have,  by \eqref{odd},
 \begin{equation}\label{P}
\mu_t=\Phi(\mu^L_t,\mu^R_t).
\end{equation}
For instance, $\mu^L_3=\mu_2=\mu_1$, and $\mu^R_3=\mu_1$. Fix $w\in\{L,R\}^\N$. With notation \eqref{www}: 
$$\mu_3=\mu_{w(0)}=\Phi(\mu^L_{w(0)},\mu^R_{w(0)})=\Phi(\mu_1,\mu_1).$$

Further dynamics is governed by the two mappings: 
$$\Phi_L, \Phi_R\colon \Delta_0(\Z)\times \Delta_0(\Z) \to \Delta_0(\Z)\times \Delta_0(\Z),$$ where
$$\Phi_L(\mu,\nu)=(\Phi(\mu,\nu),\nu)\mbox{ and }\Phi_R(\mu,\nu)=(\mu,\Phi(\mu,\nu)).$$
Equations \eqref{even} and \eqref{odd} imply that the sequence  $(\mu^L_{w(n)},\mu^R_{w(n)})_{n=0}^{\infty}$ is the trajectory of the following nonautonomous dynamical system:
\begin{equation}\label{System}
(\Delta_0(\Z)\times \Delta_0(\Z), \{\Phi_{w_n}\}_{n\geq1}),
\end{equation}
see below.
\begin{proposition}\label{TT}
For any $w=(w_1,w_2,\dots)\in\{L,R\}^\N$,
 $$(\mu^L_{w(n+1)},\mu^R_{w(n+1)})=\Phi_{w_{n+1}}\left((\mu^L_{w(n)},\mu^R_{w(n)})\right),\ n\geq 0.$$
 In other words,  $(\mu^L_{w(n)},\mu^R_{w(n)})_{n=0}^{\infty}$ is the trajectory of \eqref{System} which starts from $(\mu_1,\mu_1)$. 
\end{proposition}

\begin{proof}
Fix $w\in\{L,R\}^\N$, and $n\in\N$. Assume that $w(n)=2s+1\geq 3$ so we have $\mu^L_{w(n)}=\mu_{s+1}$ and $\mu^R_{w(n)}=\mu_{s}$. 
The cases $w_{n+1}=L$ and $w_{n+1}=R$ are similar, and we will  assume that $w_{n+1}=L$. We have:  $w(n+1)=L(2s+1)=4s+1=(2s+1)+2s$, and hence 
$$\mu_{w(n+1)}^L=\mu_{2s+1}=\mu_{w(n)}\mbox{ and }\mu_{w(n+1)}^R=\mu_s=\mu_{w(n)}^R.$$
 Thus, by \eqref{P} and by the definition of $\Phi_L$, $(\mu_{w(n+1)}^L,\mu_{w(n+1)}^R)=\Phi_L(\mu_{w(n)}^L,\mu_{w(n)}^R).$
\end{proof}

Proposition \ref{TT} allows us to write shortly:
$$\mu_{w(n)}=\Phi_{w(n)}(\mu_1), n\geq 0,$$
where  $\Phi_{w(n)}\colon \Delta_0(\Z)\to \Delta_0(\Z)$ is defined by:
$$\Phi_{w(0)}(\mu)=\Phi(\mu,\mu)\mbox{ and }   \Phi_{w(n)}(\mu):=\Phi\circ\Phi_{w_n}\circ\dots\circ\Phi_{w_1}(\mu,\mu).$$

For convenience, in next sections we will focus on the probability measures:
\begin{equation}\label{Pt}
P_{w(n)} = \Phi_{w(n)}(\delta_0),\ n\geq 0, 
\end{equation} 
rather than on $\mu_{w(n)}= \Phi_{w(n)}(\mu_1)$, which is justified by Observation \ref{O2} and equation \eqref{dc}. 
 Given $\mu,\nu\in\Delta(\Z)$, let $\mu\ast\nu$ denote the convolution of measures $\mu$ and $\nu$, i.e. :
$$(\mu\ast\nu)(d)=\sum\limits_{k\in\Z}\mu(k)\cdot\nu(d-k)=\sum\limits_{k\in\Z}\mu(k)\cdot (\sigma_{k}\nu)(d).$$
\begin{observation}\label{O2}
Fix $w\in\{L,R\}^\N$, and let $\mu\in \Delta_0(\Z)$. We have:
  \begin{equation}\Phi_{w(n)}(\mu)=\mu\ast \Phi_{w(n)}(\delta_0),\ n\geq0.
\end{equation}
\end{observation}
\begin{proof}
 Let us write $\mu=\sum\limits_{k\in\Z}p_k\delta_k$. The mapping $\Phi_{w(n)}\colon\Delta_0(\Z)\to\Delta_0(\Z)$, defined by compositions of linear combinations of shifts,  is linear, continuous, and commutes with shifts. Hence, by $\sigma_{k}\dd_0=\dd_k$,
$$\Phi_{w(n)}\left(\sum\limits_{k\in\Z}p_k\delta_k\right)=\sum\limits_{k\in\Z}p_k\cdot\Phi_{w(n)}(\delta_k)=\sum\limits_{k\in\Z}p_k\cdot\Phi_{w(n)}(\sigma_{k}\dd_0)=$$
$$=\sum\limits_{k\in\Z}p_k\cdot \sigma_{k}(\Phi_{w(n)}(\dd_0))=\sum\limits_{k\in\Z}\mu(k)\cdot \sigma_{k}(\Phi_{w(n)}(\dd_0))=\mu\ast\Phi_{w(n)}(\dd_0).$$
\end{proof}

\begin{corollary}\label{RR}
For any $w\in\{L,R\}^\N$,  by Observation \ref{O2}, $\mu_{w(n)}=\mu_1\ast P_{w(n)},$
where $P_{w(n)}=\Phi_{w(n)}(\dd_0)$. By $\eqref{mu1}$,
\begin{equation}\label{cd}
\mu_{w(n)}(d)=\sum\limits_{k\leq 1} \left(\frac12\right)^{2-k}\cdot P_{w(n)}(d-k),\ d\in\Z.
\end{equation}
\end{corollary}
 One may check by direct computation that the deconvolution formula is:
\begin{equation}\label{dc}
P_{w(n)}(d)= 2\cdot\mu_{w(n)}(d+1)-\mu_{w(n)}(d+2),\ d\in\Z.
\end{equation}
Equations \eqref{cd} and \eqref{dc} allow us to switch between $\mu_t$ and $P_t$ when necessary. The definition of $P_t$ given by $\eqref{Pt}$ is consistent with the following system:
\begin{equation}\label{system}
P_1=\dd_0,\ P_{2t}=P_t,\ P_{2t+1}=\Phi(P_{t+1},P_t),
\end{equation}
by which we may consider $\{P_t\}_{t\in\N}$ instead of $\{P_t\}_{t\in T}$ whenever convenient.

\section{Growing trees}\label{trees}

Compared to the standard dynamics behind equations \eqref{even} and \eqref{odd}  ( Proposition 2.5 in \cite{Emme}), the recursion given by Proposition \ref{TT} represents the process of growing a tree. We will start from the natural connection between the planar binary trees, \cite{Stanley}, and the bounded stopping times, \cite{Doob}. We will show that measures $P_{t}$ given by \eqref{Pt} are naturally embedded in the simple symmetric random walk $S_n$, namely, for any $t\in T$ there is a finite stopping time $\tau_t$ such that $P_t$ is the probability distribution of $S_{\tau_t}$.  A reader interested in the embedding problems related to the simple random walk is referred to \cite{Cox}, see also the original Skorohod embedding problem, \cite{Skorohod}.
\subsection{Assumptions and notation.}
 Let $\TT$ be the smallest set satisfying:
\begin{enumerate}
\item[\textup{(i)}] $\bullet \in \TT$,
\item[\textup{(ii)}] if $T_-, T_+ \in \TT$, then $[T_-, T_+] \in \TT$. 
\end{enumerate}

The elements of $\TT$ are \emph{full planar binary trees},  to which we shall refer simply as trees.. The \emph{height} of a tree $T \in \TT$, denoted by $|T|$, is defined
by : $|\bullet| = 0$, and  $|[T_-, T_+]| = \max\{|T_-|, |T_+|\} +1.$
The set of binary trees of height at most~$n$ will be denoted by
\[
  \TT_n = \{ T \in \TT : |T| \leq n \}
\]

The  set $\TT$ is partially ordered by the following relation:  

\begin{enumerate}
\item[\textup{(i)}] $\bullet \sqsubseteq T$ for every
  $T \in \TT$.
\item[\textup{(ii)}] $[S_-, S_+] \sqsubseteq [T_-, T_+]$ if and
  only if $S_- \sqsubseteq T_-$ and $S_+ \sqsubseteq T_+$.
\end{enumerate}

Roughly speaking, $S \sqsubseteq T$ iff $S$ is obtained from $T$ by replacing some subtrees with "$\bullet$" (with leaves). \\



 From now on, assume that the triple $(\Omega, \FF, \PP)$ is the canonical probability space with the balanced Bernoulli measure:
\begin{itemize}
\item $\Omega=\{-1,1\}^\N$,
\item $\FF$ is the $\sigma-$algebra of cylinder sets
\item $\PP=\bigotimes_{n=1}^{\infty} P_n$ is the product of $P_n=\frac12\dd_{-1}+\frac12\dd_1$
\end{itemize}

 The probability distribution of a random variable $X\colon\Omega\to\Z$ will be denoted by $\LL(X)$, that is, $\LL(X)\in\Delta(\Z)$ is given by: $\LL(X)(d)=\PP[X= d]$, $d\in\Z$. 

 Let $\xi_i\colon\Omega\to\{-1,1\}$, $i\in\N^+$, denote the coordinate variables: $\xi_i(\omega)=\omega_i$, $\omega\in\Omega$. Obviously $\{\xi_i\}_{i=1}^\infty$ are independent, with:
$$\PP[\xi_i = +1] = \PP[\xi_i = -1] = \tfrac{1}{2},\ i\in\N^+.$$  
 The simple random walk $\{S_n\}_{n\in\N}$ is defined by:
$$S_0=0\mbox{ and }S_n=\sum\limits_{i=1}^n\xi_i, \ n\in\N^+.$$
 The natural filtration of the process $\xi=(\xi_i)_{i\in\N}$ will be denoted by:
$$\FF_0 = \{\varnothing, \Omega\}\mbox{ and } \FF_n = \sigma(\xi_1, \ldots, \xi_n)\mbox{ for }n \geq 1.$$

A random variable $\tau \colon \Omega \to \{0, 1, 2, \dots\}$ is a \emph{stopping time} with respect to $(\FF_n)_{n \geq 0}$ iff $\{\tau \leq n\} \in \FF_n$ for every $n \geq 0$. In other words,  $\tau \colon \{-1,1\}^{\N}\to \N$ is a stopping time if $\tau$ is a Borel function with the 
following property:  if $\omega\in\{-1,1\}^\N$ and 
$\tau(\omega) = k$, then $\tau(\omega')=k$ for any $\omega'\in\{-1,1\}^{\N}$ with $(\omega_1,\dots,\omega_k)=(\omega'_1,\dots,\omega'_k)$. For the set of stopping times bounded by~$N\in\N$ we will write:
\[
  \mathcal{S}_N = \{ \tau\colon\Omega\to\N |\ \tau \text{ is a stopping time with }
    \tau \leq N \}
\]


There is a natural bijection between  trees $\TT$ and bounded stopping times, see below.

\begin{definition}\label{tt}
Given $T \in \TT$, define the stopping time $\tau_T$ recursively:
\begin{enumerate}
\item[\textup{(i)}] If $T = \bullet$, then $\tau_T = 0$.
\item[\textup{(ii)}] If $T = [T_-, T_+]$, then
\[
  \tau_T (\omega_1,\omega_2,\dots)= 1 +
  \begin{cases}
    \tau_{T_-}(\omega_2,\omega_3,\dots)& \text{if } \omega_1 = -1,\\[2pt]
    \tau_{T_+}(\omega_2,\omega_3,\dots) & \text{if } \omega_1 = +1,
  \end{cases}
\]
\end{enumerate}
\end{definition}

 For any $n\in\N$, the map  $\TT_n\ni T \mapsto \tau_T\in\mathcal{S}_n$  is a bijection between $\TT_n$ and $\mathcal{S}_n$. Furthermore, by simple induction with respect to the height of a tree it is easy to show that $S\sqsubseteq T$ implies  $\tau_S\leq\tau_T$ (the pathwise inequality: $\tau_S(\omega)\leq\tau_T(\omega)$, $\omega\in\Omega$).
Given a stopping time $\tau\in \mathcal{S}_N$, the random variable $S_{\tau}$ is given by:
$$S_{\tau}=\sum\limits_{k=0}^{N}S_k\cdot 1_{\{\tau=k\}},$$
where $1_A\colon\Omega\to\{0,1\}$ stands for the characteristic function of  the set $A\in \FF$.
Finally, we note that $\tau=\tau\circ  (\xi_1,\xi_2,\dots)$, and we will often use notation:
$$\tau(\xi_1,\xi_2,\dots):=\tau\circ (\xi_1,\xi_2,\dots).$$

\begin{example}\label{Tree3}
Let $T = [T_-, T_+] = [\bullet, [\bullet, \bullet]]$. We have
$\tau_{T_-} = 0$, $\tau_{T_+} = 1$, and
$$\tau_T \ \ = 1 +
\begin{cases}
0 & \text{if } \xi_1 = -1,\\[2pt]
1 & \text{if } \xi_1 = +1
\end{cases}. $$
\end{example}
In other words: $\tau_T=1_{\{-1\}}(\xi_1)+2\cdot 1_{\{1\}}(\xi_1)$, see the illustration: 
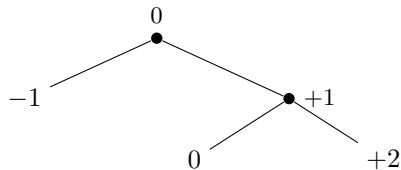
\begin{center}
\begin{tikzpicture}[
level distance=8mm,
level 1/.style={sibling distance=35mm},
level 2/.style={sibling distance=25mm},
dot/.style={circle, fill=black, inner sep=1.5pt},
leaf/.style={font=\normalsize},
edge from parent/.style={draw, thin},
]
\node[dot, label=above:{$0$}] {}
child {node[leaf] {$-1$}}
child {node[dot, label=right:{$+1$}] {}
child {node[leaf] {$0$}}
child {node[leaf] {$+2$}}
};
\end{tikzpicture}
\end{center}

\begin{definition}[Embedding map]\label{def:embedding}
The \emph{embedding map} $\mathcal{E} \colon \TT \to \Delta(\Z)$ assigns to each tree
 the law of the corresponding stopped random walk:
\[
  \mathcal{E}(T) = \mathcal{L}(S_{\tau_T}).
\]
\end{definition}
\subsection{Growing trees} We will now construct the family of trees $\{T_t\colon t\in 2\N+1\}\subset \TT$ such that $s\preceq t$ implies 
$T_s\sqsubseteq T_t$, and the family  $\{P_t\colon t\in 2\N+1\}$ given by \eqref{Pt} satisfies:
$$P_t=\E(T_t).$$

We have:
$$P_1=\dd_0\mbox{ and }P_3=\Phi(\dd_0, \dd_0)=\frac{1}{2}\dd_{-1} + \frac{1}{2}\dd_{1},$$
and
$$\mathcal{E}(\bullet)=\LL(S_0)= \delta_0=P_1\mbox{ and }\mathcal{E}([\bullet,\bullet])=\LL(S_1)=P_3.$$
Thus, define $T_1 = \bullet$ and $T_3 =[T_3^-,T_3^+]=[\bullet,\bullet]$, so  $\tau_1:=\tau_{T_1}=0$ and $\tau_3:=\tau_{T_3}=1$ satisfy:
$$P_1=\mathcal{L}(S_{\tau_1})\mbox{ and }P_3=\mathcal{L}(S_{\tau_3}).$$

 Further construction proceeds by growing trees according to the following recursion: at each step, one subtree is left intact while the other subtree grows to become a copy of the current tree, and the 
successive letters from $\{L,R\}$ determine which subtree  is left intact and which one grows at each step. So far we have 
$$T_3=[T_3^-,T_3^+],\mbox{ where }T_3^+=T_3^-=T_1=\bullet.$$


\textbf{Recursive step.} 
Given $v\in T\setminus\{1\}$ and
$$T_v = [T_v^-, T_v^+] \in \TT,$$ define:
\begin{equation}\label{td}
T_{vL} = [T_v,\, T_v^+],\mbox{ and }T_{vR} = [T_v^-,\, T_v].
\end{equation}

By the construction, if $s\preceq t$ then $T_s\sqsubseteq T_t$. To verify that the recursive step leads to $\mathcal{E}(T_v) = P_{v}$, $v\in\{L,R\}^\star$, it is enough to show that $\E(T_{w(n)}) = P_{w(n)}$, $n\in\N$, for any sequence $w\in\{L,R\}^{\N}$. In other words, it is enough to check that \eqref{td}  matches the recursion from Proposition \ref{TT}.
\begin{theorem}\label{RE}
For any $v\in\{L,R\}^{\star}$, $\mathcal{E}(T_v) = P_{v}$. In other words: $\mathcal{L}(S_{\tau_v}) = P_{v}$, where \ $\tau_v = \tau_{T_v}$.
\end{theorem}
\begin{proof}
For $v=\varepsilon$ and for $P_3^L:=P_1=\E(T_3^-)$,\  $P_3^R:=P_1=\E(T_3^+)$, we have
$$P_3=\mathcal{E}([T^-_3,T^+_3])\mbox{ and } P_3=\Phi(\E(T_3^-),\E(T_3^+))=\Phi(P_3^L,P_3^R).$$
 Now, assume that for some $3\preceq v$ we have:
$$P_v=\mathcal{E}([T^-_v,T^+_v])\mbox{ and } P_v=\Phi(\E(T_3^-),\E(T_3^+))=\Phi(P_v^L,P_v^R),$$
where $P_v^L=\E(T_v^-)$ and $P_v^R=\E(T_v^+)$ (considering the extension \eqref{system}, explicitly: $P_{2s+1}^L=P_{s+1}$ and  $P_{2s+1}^R=P_{s}$). Note that  $\sigma_L(\LL(X))=\LL(X-1)$ and $\sigma_R(\LL(X))=\LL(X+1)$ for any $r.v.$ $X\colon\Omega\to\Z$, where $\sigma_L$, $\sigma _R$ are given by \eqref{SLR}.  By  Definition \ref{tt}, for any $T_-,T_+\in\TT$,
\begin{equation}\label{e} \mathcal{E}([T_-, T_+])=\mathcal{L}(S_{\tau_{[T_-, T_+]}})=  \tfrac{1}{2}\,\mathcal{L}(-1 + S_{\tau_{T_-}}) 
     +\tfrac{1}{2}\,\mathcal{L}(1 + S_{\tau_{T_+}})=\end{equation}
$$ =\tfrac{1}{2} \sigma_L\bigl(\mathcal{E}(T_-)\bigr)
    +   \tfrac{1}{2} \sigma_R\bigl(\mathcal{E}(T_+)\bigr)=\Phi(\E(T_-),\E(T_+)).$$
Above, we have used the fact that for $ j\in \{-,+\}$ the law $\LL(S_{\tau_{T_j}})$ is the same as the law of $S_{\tau_{T_j}}'=S_{\tau_{T_j}}\circ (\xi_2,\xi_3,\dots)$ which is defined by the 
same formula applied to the shifted sequence. As $P_{vL}=\Phi(P_v,P_v^R)$ and $P_{vR}=\Phi(P_v^L,P_v)$,  equation \eqref{e} guarantees that 
$$P_{vL}=\E(T_{vL})\mbox{ and }P_{vR}=\E(T_{vR}).$$
In other words, for any $w\in\{L,R\}^\N$ the recursive step \eqref{td} matches the dynamics of the sequence 
$P_{w(n)}$, ${n\in\N}$, that is: $P_{w(n)W}=\Phi\circ\Phi_W(P_{w(n)}^L,P_{w(n)}^R)$, where $W\in\{L,R\}.$
\end{proof}



The explicit recursive description of the stopping times $\tau_{v}(\xi_1,\xi_2,\dots)$ goes by: if \[
  \tau_v(\xi_1,\xi_2,\dots) = 1 +
  \begin{cases}
    \tau^L_{v}(\xi_2,\xi_3,\dots) & \text{if } \xi_1 = -1,\\[4pt]
    \tau^R_{v}(\xi_2,\xi_3,\dots) & \text{if } \xi_1 = +1,
  \end{cases}
\]
then: 
\begin{equation}\label{stop}  \tau_{vL}(\xi_1,\xi_2,\dots)=1 +
    \begin{cases}
     \tau_v(\xi_2,\dots)&\text{if } \xi_1= -1
       \\
       \tau_{v}^R(\xi_2,\dots)&\text{if } \xi_1 = +1
     
  \end{cases},
\end{equation}
and the  formula for $\tau_{vR}$ is analogous. The construction implies that: $T_w \sqsubseteq T_{wL}$ and  $T_w \sqsubseteq T_{wR}$, and  hence: $\tau_{w}\leq\tau_{wL}\mbox{ and }\tau_{w}\leq\tau_{wR}$.
 Moreover, the height of the tree $|T_w|=\max(\tau_w)$ increases by at most one at each step and hence $\tau_w\leq \ell(w)+1$, where $\ell(w)$ denotes the length of the word $w$.  The dynamics that governs the sequence $\{T_t\}_{t\in T}$ determines  a hierarchical structure exhibiting self-similarity. Let us take a look on a few pictures which illustrate  the growth of $T_{LRLL}$, a tree which determines measure $P_{41}=P_{LRLL}$.

\begin{example}\label{ex:trees_3_5}
The measure $P_3 = \frac{1}{2}\delta_1 + \frac{1}{2}\delta_{-1}$
corresponds to the tree $T_3 = [\bullet,\bullet]$ on the left.  Measure $P_5=P_{L}=\mathcal{E}(T_{L}) = \frac{1}{2}\delta_1 + \frac{1}{4}\delta_0
+ \frac{1}{4}\delta_{-2}$ is represented by the tree $T_{L}=[T_3,T_3^+]=[[\bullet,\bullet],\bullet]$ on the right. The left branch of $T_3$ has grown to copy the $T_3$ while the right branch of $T_3$ has stayed intact. 

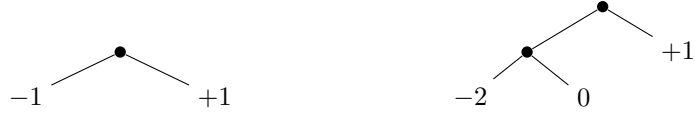
\begin{center}
\begin{tikzpicture}[
  level distance=6mm,
  level 1/.style={sibling distance=25mm},
  level 2/.style={sibling distance=20mm},
  dot/.style={circle, fill=black, inner sep=1.5pt},
  leaf/.style={font=\normalsize},
  edge from parent/.style={draw, thin},
]
\node[dot] {}
  child {node[leaf] {$-1$}}
  child {node[leaf] {$+1$}};
\end{tikzpicture}
\hspace{25mm}
\begin{tikzpicture}[
  level distance=6mm,
  level 1/.style={sibling distance=20mm},
  level 2/.style={sibling distance=15mm},
  dot/.style={circle, fill=black, inner sep=1.5pt},
  leaf/.style={font=\normalsize},
  edge from parent/.style={draw, thin},
]
\node[dot] {}
  child {node[dot] {}
    child {node[leaf] {$-2$}}
    child {node[leaf] {$0$}}
  }
  child {node[leaf] {$+1$}};
\end{tikzpicture}
\end{center}
\end{example}

\begin{example}%
\label{PLR}
$P_{11}=P_{LR}=\mathcal{E}(T_{LR}) = \frac{1}{4}\delta_2 + \frac{1}{8}\delta_1
+ \frac{1}{4}\delta_0 + \frac{1}{8}\delta_{-1}
+ \frac{1}{4}\delta_{-2}$, where $T_{LR} = [T_L^-,\, T_L]$.
The left subtree $T_L^- = [\bullet, \bullet]$ is left intact; the right subtree of $T_{LR}$ is a copy of the entire tree $T_L$. Note that the resulting measure $P_{11}$ is symmetric around zero.
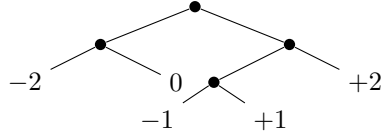
\begin{center}
\begin{tikzpicture}[
  level distance=5mm,
  level 1/.style={sibling distance=25mm},
  level 2/.style={sibling distance=20mm},
  level 3/.style={sibling distance=15mm},
  dot/.style={circle, fill=black, inner sep=1.5pt},
  leaf/.style={font=\normalsize},
  edge from parent/.style={draw, thin},
]
\node[dot] {}
  child {node[dot] {}
    child {node[leaf] {$-2$}}
    child {node[leaf] {$0$}}
  }
  child {node[dot] {}
    child {node[dot] {}
      child {node[leaf] {$-1$}}
      child {node[leaf] {$+1$}}
    }
    child {node[leaf] {$+2$}}
  };
\end{tikzpicture}
\end{center}
\end{example}

\begin{example}%
\label{PLRL}

 $P_{21}=P_{LRL}=\mathcal{E}(T_{LRL})= \tfrac{1}{4}\delta_2 + \tfrac{1}{4}\delta_1
  + \tfrac{1}{16}\delta_0 + \tfrac{1}{4}\delta_{-1}
  + \tfrac{1}{16}\delta_{-2} + \tfrac{1}{8}\delta_{-3}$, where
$T_{LRL} = [T_{LR},\, T_{LR}^+] = [T_{LR},\, T_L]$.

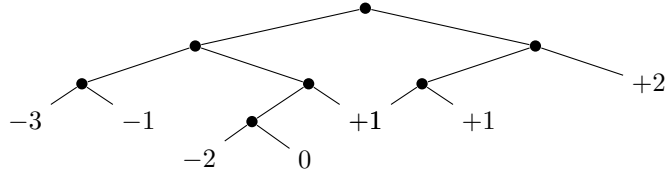
\begin{center}
\begin{tikzpicture}[
  level distance=5mm,
  level 1/.style={sibling distance=45mm},
  level 2/.style={sibling distance=30mm},
  level 3/.style={sibling distance=15mm},
  level 4/.style={sibling distance=14mm},
  dot/.style={circle, fill=black, inner sep=1.5pt},
  leaf/.style={font=\normalsize},
  edge from parent/.style={draw, thin},
]
\node[dot] {}
  child {node[dot] {}
    child {node[dot] {}
      child {node[leaf] {$-3$}}
      child {node[leaf] {$-1$}}
    }
    child {node[dot] {}
      child {node[dot] {}
        child {node[leaf] {$-2$}}
        child {node[leaf] {$0$}}
      }
      child {node[leaf] {$+1$}}
    }
  }
  child {node[dot] {}
    child {node[dot] {}
      child {node[leaf] {$-1$}}
      child {node[leaf] {$+1$}}
    }
    child {node[leaf] {$+2$}}
  };
\end{tikzpicture}
\end{center}
\end{example}
\begin{example} Below we represent $P_{41}=\E(T_{LRLL})$, where $T_{LRLL}=[T_{LRL},T_L]$.
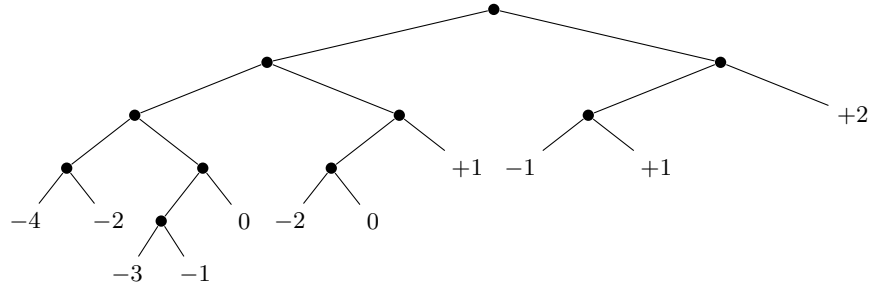
\begin{center}
\begin{tikzpicture}[
  level distance=7mm,
  level 1/.style={sibling distance=60mm},
  level 2/.style={sibling distance=35mm},
  level 3/.style={sibling distance=18mm},
  level 4/.style={sibling distance=11mm},
  level 5/.style={sibling distance=9mm},
  dot/.style={circle, fill=black, inner sep=1.5pt},
  leaf/.style={font=\small},
  edge from parent/.style={draw, thin},
]
\node[dot] {}
  child {node[dot] {}
    child {node[dot] {}
      child {node[dot] {}
        child {node[leaf] {$-4$}}
        child {node[leaf] {$-2$}}
      }
      child {node[dot] {}
        child {node[dot] {}
          child {node[leaf] {$-3$}}
          child {node[leaf] {$-1$}}
        }
        child {node[leaf] {$0$}}
      }
    }
    child {node[dot] {}
      child {node[dot] {}
        child {node[leaf] {$-2$}}
        child {node[leaf] {$0$}}
      }
      child {node[leaf] {$+1$}}
    }
  }
  child {node[dot] {}
    child {node[dot] {}
      child {node[leaf] {$-1$}}
      child {node[leaf] {$+1$}}
    }
    child {node[leaf] {$+2$}}
  };
\end{tikzpicture}
\end{center}

\end{example}

\section{The martingale evolution of $P_t$}\label{martingale} 
We proceed under the notation and the assumptions of the previous sections. We have established: $s\preceq t \Rightarrow \tau_s\leq\tau_t$, where $P_{t}=\LL(S_{\tau_t})$. For  $w\in\{L,R\}^\N$, define:
$$\eta^w_t:=\tau_{w(t)}, \ \ \  t\in\Z_{\geq -1},$$
so we have: $\eta^w_{-1}=0$,  $\eta^w_0=1$. Define:
\begin{equation}\label{XM}
X^w_{t}=S_{\eta^w_t},\ t\in\Z_{\geq -1}.
\end{equation}
When $w\in\{L,R\}^\N$ is fixed, we will often write simply: $\eta_t=\eta^w_t$ and $X_t=X_t^w$. 

\subsection{Martingale property and Wald identities} 
Fix $w\in\{L,R\}^\N$.
As we have shown in the previous section, for $0\leq s<t$ we have 
$$\eta_s=\tau_{w(s)}\leq\tau_{w(t)}=\eta_t\leq \ell(w_1\dots w_{t})+1=t+1,$$
  and thus, by Doob's optional stopping theorem, $X_t=S_{\eta_t}$ forms a martingale (more precisely, a martingale with respect to the stopped filtration, which forces that $X_t$ is a martingale with respect to its natural filtration $F_t=\sigma(X_t,\dots,X_1)$).  Additionally, by the second Wald's identity, the variance of the measure $P_t$ equals to the expected number of steps of  the stopping time $\tau_{t}$ which defines the measure $P_t$ by $P_t=\LL(S_{\tau_t})$. 
\begin{theorem}
For any $w\in\{L,R\}^\N$, the  process  $X_t=X_t^w$ given by \eqref{XM} is a martingale. In particular, $E[X_t|X_s]=X_{s},\mbox{ for }t\geq s.$
By Wald's identities, 
$$E[X_t]=0\mbox{ and }D^2[X_t]=E[\eta_t].$$ 
\end{theorem}
\begin{proof}
  
We will prove briefly  the second statement. Both processes $\{S_n\}_{n\in\mathbb{N}}$ and $\{S_n^2-n\}_{n\in\N}$ are martingales with respect to $F_n=\sigma(\xi_1,\dots,\xi_n)$, and the stopping times $\eta_t$ are bounded. Hence, by Doob's theorems: $E[S_{\eta_t}]=E[\xi_1]\cdot E[\eta_t]=0$, and $E[(S_{\eta_t})^2-\eta_t]=0$, which translates into $D^2[S_{\eta_t}]=E[\eta_t]$ 
\end{proof}

\subsection{Support} Given $P\in\Delta(\Z)$,  let $\supp(P)=\{d\in\Z\colon P(d)>0\}$.
 Given $w=w_1\dots w_n\in\{L,R\}^\star$, let 
$$\ell_L(w)=\sum\limits_{i=1}^n 1_{\{L\}}(w_i)\mbox{ and } \ell_R(w)=\sum\limits_{i=1}^n 1_{\{R\}}(w_i).$$
 We have $\supp(P_{\eps})=\supp(P_3)=\{-1,1\}$. Fix $w\in \{L,R\}^\star\setminus\{\eps\}$. We have $\max(\tau_w)=\ell(w)+1$. Denote $\tau_w\wedge n:=\min(\tau_w,n)$. The stopped process $\{S_{\tau_w\wedge n}\}_{n=0}^{\ell(w)+1}$ takes at most $|w|_R+1$ steps to the right  and $|w|_L+1$ steps to the left (this follows from the recursive definition, \eqref{td}, \eqref{stop}).  Furthermore: $ \{S_{\tau_w}=|w|_R+1\}=\{\xi_1=1,\dots,\xi_{|w|_R+1}=1\}$ and $\{S_{\tau_w}=-(|w|_L+1)\}=\{\xi_1=-1,\dots,\xi_{|w|_L+1}=-1\}.$ Hence, $$\min(\supp(P_{w}))=-(|w|_L+1)\mbox{ and }\max(\supp(P_{w}))=|w|_R+1,$$
and
\begin{equation}\label{PWLR}
P_{w}(-(|w|_L+1))=\left(\frac{1}{2}\right) ^{|w|_L+1}\mbox{ and }P_{w}(|w|_R+1)=\left(\frac{1}{2}\right)^{|w|_R+1}.
\end{equation}

\subsection{Symmetries.}\label{symsym}

 Given $w\in\{L,R\}^\star$, we will write $\overline{w}$ for the word obtained from $w$ by interchanging $L$ and $R$ (for example, $\overline{LLR}=RRL$).  By simple induction with respect to the length of the word $w$,  it is easy to see that $ S_{\tau_{\omega}}\stackrel{d}{=}-S_{\tau_{\overline{\omega}}}$, and hence: $$P_w(d)=P_{\overline{w}}(-d),\ d\in\Z.$$ 

Now, given $w=w_1\dots w_n$, let $\overleftarrow{w}$ denote the reversed word: $\overleftarrow{w}=w_n\dots w_1$ ( $\overleftarrow{\varepsilon}:=\varepsilon$). The following symmetry may be a bit surprising:
\begin{equation}\label{sym2}
P_w(d)=P_{\overleftarrow{w}}(d),\ w\in\{L,R\}^\star,\ d\in\Z .
\end{equation}
The above follows from the available results: paper \cite{MS} has shown that 
$\mu_{(t)_2}=\mu_{\overleftarrow{(t)}_2}$, where $\overleftarrow{(t)}_2$ reverses the digits in the binary representation $(t)_2$. Our representation $t=w_1\dots w_n$ inherits this reverse property by Observation \ref{binarr} (also, by the deconvolution formula). In particular, we see that the mapping $T\ni \to P_t\in\Delta_0(\Z)$ is not injective (the map $T\ni t\to T_t\in\TT$ is injective). The following theorem summarizes the symmetries.
\begin{theorem}\label{tsym}
For any $w\in\{R,L\}^*$, 
$$P_w(d)=P_{\overline{w}}(-d)\mbox{ and }P_w(d)=P_{\overleftarrow{w}}(d),\ d\in\Z.$$
\end{theorem}
\begin{conclusion}The alternating word $v^{2n} = (LR)^n$ of length $2n$ ($v^{2n}=LRLR...LR$) satisfies
$$\overline{v^{2n}} =  \overleftarrow{v^{2n}}=(RL)^n.$$  
 By Theorem \ref{tsym}, the above  implies:
$$P_{v^{2n}}(d)=P_{v^{2n}}(-d),\  d\in\Z.$$
\end{conclusion}
In particular,  showing $P_{v^{2n}}(\N)\to \frac12$ is equivalent to $P_{v^{2n}}(\{0\})\to 0$. The word  $v^{2n}$ plays crucial role in the variance control.

\subsection{Variance analysis}
Given $P\in\Delta(\Z)$, we will write 
$$E[P]=\int\limits_{\Z}xP(dx)\mbox{ and }D^2[P]=\int_{\Z}(x-E[P])^2P(dx).$$
Controlling the variance  of $\mu_t$ is a nontrivial problem, important for the analysis of the asymptotics of $\mu_t$, see for instance \cite{drmota},\cite{Emme},\cite{Hosten},\cite{sw2023}. We  now show that considering the poset $(T, \preceq)$ allows for very natural variance control. One of the reasons is that $(T, \preceq)$ arranges the measures $\mu_t$ in such a way that the variance is monotone. First, note that by $\mu_t=\mu_1\ast P_t$, we have $D^2[\mu_t]=2+D^2[P_t]$, and we thus focus on $P_t$.  We have:
\begin{equation}\label{monovar}
s\preceq t \Longrightarrow D^2(P_s)\leq D^2(P_{t}).
\end{equation}
Within our framework, the above is immediate:
$$s\preceq t \Longrightarrow D^2(P_s)=E[\tau_s]\leq E[\tau_{t}]=D^2(P_{t}).$$
The monotonicity \eqref{monovar} may be also concluded from the fact that for any $\mu,\nu\in \Delta_0(\Z)$
\begin{equation}\label{variance}
D^2[\Phi(\mu,\nu)]=\frac12 D^2[\mu]+\frac12 D^2[\nu]+1
\end{equation}
Define recursively: $L^{n+1}:=L^nL$ and $R^{n+1}:=R^nR$. Note that constant words $L^n$ and $R^n$ ( $L^n= 2^{n+1}+1$, $R^n=2^{n+2}-1$)  have the smallest variance among all $n$-letter words:
$$\min\limits_{w\in\{L,R\}^n}D^2(P_{w})=D^2(P_{L^n})=D^2(P_{R^n}).$$
By \eqref{variance}, $D^2[P_{L^{n+1}}]=\frac{1}{2}\cdot D^2[P_{L^n}]+1$, where $D^2[P_{L^0}]:=D^2[P_{\eps}]=D^2[\frac12\dd{-1}+\frac12\dd_1]=1$. This yields:
\begin{equation}\label{minD}
\min\limits_{w\in \{L,R\}^k} D^2[P_w]=D^2[P_{L^k}]=D^2[P_{R^k}]=2-(\frac{1}{2})^{k}.
\end{equation}
On the other hand, quick analysis of \eqref{variance} shows that the alternating word $v^n$ ( $v^n_{k-1}\neq v^n_{k}$ for $1<k\leq n$) satisfies:
$$\max\limits_{w\in \{L,R\}^n} D^2[P_w] =D^2[v^n].$$
By induction, one may show that:
$$
 \; D^2[P_{v^n}] \;=\; \frac{2n}{3} \;+\; \frac{8}{9} \;+\; \frac{(-1)^{n}}{9\cdot 2^{n}},\ n\geq0, \ ( v^0:=\eps). $$
Now, we will take a more common approach: controlling the variance by the number of blocks $\ell_b(w)$ in the word $w\in\{L,R\}^\star\setminus\{\eps\}$. More precisely, 
$$\ell_b(w_1^{n_1}w_2^{n_2}...w_k^{n_k})=k,\mbox{ where } w_i\neq w_{i+1}\mbox{ and } n_i>0.$$  
 By simple induction we obtain the monotonicity: 
\begin{equation}\label{D2}
D^2[w_1^{n_1}w_2^{n_2}...w_k^{n_k}]\leq D^2[w_1^{n_1+m_1}w_2^{n_2+m_2}...w_k^{n_k+m_k}],\mbox{ where }m_i\geq0,\ i=1,\dots,k .
\end{equation}
Again, the alternating word $v^n$ plays the fundamental role: $l_b(v^n)=n$ and, by \eqref{D2},
$$D^2[P_{v^k}]=\min\limits_{\{w\colon l_b(w)=k\}}D^2(P_w).$$

Recall that $\lim\limits_{n\to\infty}D^2[P_{L^n}]=\lim\limits_{n\to\infty}D^2[P_{R^n}]=2$. Now, given the alternating word $v\in\{L,R\}^k$, from \eqref{D2}, by induction with respect to $k$,
\begin{equation}\label{limlim}
\sup\limits_{\{w\in \{L,R\}^\star \colon \ell_b(w)=k\}}D^2[P_w]=\lim\limits_{n\to\infty} D^2[P_{v_1^nv_2^n\dots v_k^n }]=\underbrace{2 + 2 + \cdots + 2}_{k}=k\cdot 2.
\end{equation}
We summarize this subsection with the following theorem.
\begin{theorem}\label{WAR}For any $w\in\{L,R\}^\star$,
$$2-(\frac{1}{2})^{\ell(w)}\leq D^2[P_w]\leq \frac{2\ell(w)}{3} \;+\; \frac{8}{9} \;+\; \frac{(-1)^{\ell(w)}}{9\cdot 2^{\ell(w)}},$$and for any $w\in\{L,R\}^\star\setminus\{\eps\}$,
$$ \frac{2\ell_b(w)}{3} \;+\; \frac{8}{9} \;+\; \frac{(-1)^{\ell_b(w)}}{9\cdot 2^{\ell_b(w)}}\leq D^2[P_w]< 2\cdot \ell_b(w),$$
None of the above inequalities can be improved- the weak inequalities are attained by the constant word/alternating word, and the strong inequality is given by the limit $\eqref{limlim}$. 
\end{theorem}
\begin{remark}
 For any $w\in\{L,R\}^\N$, the sequence $\{P_{w(t)}\}_{t=-1}^{\infty}$ is a peacock,\cite{Hirsch}, that is, $P_{w(t)}$ are marginals of a martingal. Hence, for any convex function $\psi\colon\Z\to\R$,  
\begin{equation}\label{r}\textstyle \int\limits_{\Z}\psi d P_{w(t)}\leq \int\limits_{\Z}\psi d P_{w(t+1)}. \end{equation} 
Due to Kellerer, \cite{Kellerer}, \eqref{r} is also sufficient for the existence of the  martingale which admits $P_{w(t)}$. Property \eqref{r} forces the monotonicity of the span of the support and the monotonicity of the variance, both in the sense of non-strict inequalities. In the case of the family $P_{w(t)}$ the inequalities are strict, for instance, $diam(\supp(P_{w(t+1)}))=diam(\supp(P_{w(t)}))+1$. 
\end{remark}
\subsection{Asymptotic behaviour} 
The asymptotic behaviour of $X^w_t$  depends on whether the variance $D^2[X^w_t]$ is bounded or tends to infinity. It is shown in \cite{Hosten}, \cite{sw2023}, that the sequence $\mu_{t_k}$ satisfies the central limit theorem (CLT)  if the number of blocks in the binary representation $(t_k)_2$ goes to infinity. In this section we will put more attention to the complementary  scenario $\sup\limits_{t\in\N}D^2[X^w_t]<\infty$. In this setting, the martingale $X^w_t$ is almost surely convergent. 

\textbf{Bounded variance.} Assume that $\sup\limits_{t\in\N}D^2[P_{w(t)}]<+\infty$. This is the case of the sequences $w\in\{L,R\}^{\N}$ that are eventually constant. We have: $\sup\limits_{t\in\N}E|X^w_t|^2<\infty$, and by Doob theorem, the almost sure limit $X^w_{\infty}=\lim\limits_{t\to\infty}X^w_t$ exists, and equals to the $L^2$- norm limit. The limit can be written explicitly, and the asymptotic distributions form the family $\{\mu_t\}_{t\in T}\cup \{\hat{\mu}_t\}$,  where $\hat{\mu}_t$ are given by the reflection: $\hat{\mu}_t(d)=\mu_t(-d)$, $d\in\Z$.
\begin{theorem}\label{As}
  For any word $v\in\{L,R\}^{\star}$, $\lim\limits_{n\to\infty} P_{vRL^n}= \mu_v$ and $\lim\limits_{n\to\infty} P_{vLR^n}= \hat{\mu}_v$. 
Additionally: $\lim\limits_{n\to\infty}P_{L^n}=\mu_1$ and 
$\lim\limits_{n\to\infty}P_{R^n}=\hat{\mu}_1$. Hence:
 $$\{\LL(X^w_{\infty})\colon w \in\{L,R\}^\N\mbox{ is eventually constant}\}=\{\mu_t\}_{t\in T}\cup \{\hat{\mu}_t\}_{t\in T}.$$

\end{theorem}

\begin{proof}
 We start with the constant sequences:  $w^1=R^{\infty}$ and $w^2=L^{\infty}$. We have 
$$T_{R^{n+1}}=[\bullet, T_R^n],\ T_{L^{n+1}}=[ T_L^n,\bullet],\mbox{ where }T_{R^0}=T_{L^0}=T_{\eps}=[\bullet,\bullet]. $$
Given $k\in\Z$, let 
$$\tau(k)=\inf\{i\in\N^+\colon \xi_i=k\}.$$ The limits of stopping times $\tau_{R^n}$, $\tau_{L^n}$, are:
$$\tau_{R^{\infty}}:=\lim\limits_{n\to\infty}\tau_{R^n}=\tau(-1),\ \tau_{L^{\infty}}:=\lim\limits_{n\to\infty}\tau_{L^n}=\tau(1).$$
 The limit $X_\infty$ is explicit: $ X_\infty^{w_1}=S_{\tau_{R^{\infty}}}=S_{\tau(-1)}\mbox{ and }X_\infty^{w_2}=S_{\tau_{L^{\infty}}}=S_{\tau(1)},$
and hence
$$\LL( X_{\infty}^{w_1})=\LL(S_{\tau(-1)})=\hat{\mu_1}\mbox{ and }\LL( X_{\infty}^{w_2})=S_{\tau(1)}=\mu_1.$$
Now, fix an arbitrary word $v\in\{L,R\}^{\star}$. We have: $T_v=[T_v^-,T_v^+]$,  $T_{vR}=[T_v^-,T_v]$, $T_{vRL}=[ T_{vR} ,T_v]$. In particular: $P_{vR}=\Phi(P_v^L,P_v)$, and:
$$P_{vRL}=\Phi(P_{vR},P_v)=\frac12\cdot \sigma_{-1}(P_{vR}) + \frac12\cdot \sigma_{1}(P_v)=\frac14\cdot \sigma_{-2}(P_v^L) + \frac14\cdot \sigma_{0}(P_{v}) + \frac12\cdot \sigma_{1}(P_{v}).$$
By iterating the above, we get:
\begin{equation}\label{iterating}
P_{vRL^n}=\frac{1}{2^{n+1}}\cdot \sigma_{-(n+1)}(P_v^L) + \sum\limits_ {i=-(n-1)}^1  \frac{1}{2^{2-i}}\cdot \sigma_{i}(P_v).
\end{equation}
  The  law  $\LL(X^w_{\infty})$, where $w=vRL^{\infty}$,  is the limit of the above sum:
$$\LL(X^w_{\infty})=\sum\limits_{i=-\infty}^{1} \frac{1}{2^{2-i}}\cdot \sigma_{i}(P_v)=\mu_1 * P_{v}=\mu_{v}.$$
To show $\LL(X_{\infty}^{w_2})=\hat{\mu_{v}}$, where $w_2=vLR^{\infty}$, we can repeat the above reasoning, or just use the symmetry $P_{\overline{w}}(d)=\hat{P}(d)$.
\end{proof}
\begin{observation} Given  $v\in\{L,R\}^\star$, one may note that $\tau_{vRL^n}\to \tau$ a.s., where:
$$\tau = \tau(1) +  \tau_v(\xi_{\tau_1+1},\xi_{\tau_1+2},\cdots)=\tau(1)+\tau_v\circ \sigma_{\tau(1)}\circ (\xi_1,\xi_2,\dots),$$
where $\sigma_{d}$ is the right shift on $\{-1,1\}^\N$. In words, the stopping rule for $\tau=\lim\limits_{n\to\infty}\tau_{vRL^n}$ is  given by: wait for the first $i\in\N^+$ with $\xi_i=1$, and next 
switch the stopping rule to $\tau_v$. Additionally, by the strong Markov property of $S_n$, 
$$\LL(S_{\tau})=\LL\left(\sum\limits_{i=1}^{\tau(1)}\xi_i+\sum\limits_{i=\tau(1)+1}^{ \tau_v\circ \sigma_{\tau(1)}(\xi_1,\xi_2,\dots)}\xi_i\right)=\LL(S_{\tau(1)}) \ast \LL( S_{\tau_{v}})=\mu_1\ast P_{v}.$$
\end{observation}
\textbf{Unbounded variance}. If $w\in\{L,R\}^\N$ satisfies $l_b(w_1\dots w_t)\to+\infty$, then, after the renormalization, the sequence $\mu_{w(t)}$ converges weakly to the standard normal distribution:
 $$\frac{\mu_{w(t)}}{\sqrt{D^2[\mu_{w(t)}]}}\to N(0,1),$$
 see \cite{sw2023}, \cite{Hosten} for more details. As $\mu_t=\mu_1\ast P_t$, and $D^2(\mu_1)=2<\infty$, we get: 
$$\frac{P_{w(t)}}{\sqrt{D^2[P_{w(t)}]}}\to N(0,1).$$
We note that there are various results on the central limit theorem for martingales, \cite{Hall}, and stopped random walks, \cite{Gut}, although this topic is beyond the scope of the present paper.

\section{The tree asymmetry  and Cusick conjecture.}\label{CUS}
Considering the dynamics behind the equations \eqref{system}, it is natural to conjecture:
\begin{equation}\label{medP}
P_t(\N)\geq\frac12,\ t\in\N.
\end{equation}
Equivalently: $P_t(-\N)\geq\frac12$, $t\in\N$, by the symmetries. The proof of \eqref{medP} is not obvious. It is worth to mention that conjecture \eqref{medP} is presented in Section 3.4 of \cite{drmota} along with numerical verification, and the  authors note that \eqref{medP} implies $\mu_t(\N)\geq\frac12$ and $\mu_t(-\N)\geq\frac12$ (Lemma 5 in \cite{drmota}). Within our framework, this observation is the property of the limit:
$$\mu_t(\N)=\lim\limits_{n\to\infty}P_{tRL^n}(\N)\mbox{ and }\hat{\mu}_t(\N)=\lim\limits_{n\to\infty}P_{tLR^n}(\N),$$
and equation \eqref{medP} is equivalent to the median-preserving property of the martingale:
\begin{equation}\label{medPP}
\PP[X^w_t\geq 0]\geq\frac12\mbox{ and }\PP[X^w_t\leq 0]\geq\frac12,\mbox{ for  }w\in\{L,R\}^\N,\ t\in\N.
\end{equation}
If the above holds true, then by our hierarchical construction the stopped random walk $X^w_t$ splits the mass of every node of the tree in such a way that the median of the attached subtree is preserved.  Hence,  by \eqref{PWLR} and by the construction:
\begin{observation}
If \eqref{medPP} holds true, then for any $v\in\{L,R\}^\star$, for any natural $k\leq |v|_R+1$, $l\leq |v|_L+1$, we have: $P_v[\Z_{\geq k}]\geq (\frac{1}{2})^{k+2}$ and $P_v[\Z_{\leq -l}]\geq (\frac{1}{2})^{l+2}$.
\end{observation}
The  weak inequalities \eqref{medP}/\eqref{medPP} seem to be not sufficient to force the strong inequality from the Cusick problem. Based on the tree dynamics developed in this paper, we propose a natural generalization of Cusick conjecture. Fix $w\in\{L,R\}^\N$. At the beginning, the tree $T_{w(0)}=[\bullet,\bullet]$ is symmetric, and $P_{w(0)}=\frac12\delta_{-1}+\frac12\delta_1$. Next, the first letter introduces the asymmetry:
$$P_{L}(\N)=\frac34>\frac12\mbox{ and }P_{R}(-\N)=\frac34>\frac12.$$
We conjecture that the above asymmetry persists during the whole evolution process:
\begin{conjecture}\label{Azyan} For any $v\in\{L,R\}^{\star}$,
\begin{equation}\label{concon}
P_{Lv}(\N)>\frac12\mbox{ and }P_{Rv}(-\N)>\frac12,.
\end{equation}
\end{conjecture}
In words: once the tree begins to grow, one side becomes heavier than $\frac12$, and it remains so  throughout the whole process. The Cusick conjecture is the special case of this claim:
\begin{lemma}
If $P_{Lw}(\N)>\frac12$, $w\in\{L,R\}^\star$, then $\mu_v(\N)>\frac12$ for any $v\in \{L,R\}^\star$.
\end{lemma}
\begin{proof}
Fix $v\in \{L,R\}^\star$. We have $\mu_v(\N)=\lim\limits_{n\to\infty}P_{vRL^n}(\N)$, and, by \eqref{iterating},
$$P_{vRL^n}(\N)=(\frac12)^{n+1} \sigma_{(-n-1)}(P^L_v)(\N)+(\frac12)^{n+1} \sigma_{(1-n)}(P_v)(\N) + \dots+(\frac12)^2 \sigma_{0}(P_v)(\N)+\frac12 \sigma_{1}(P_v)(\N).$$
Additionally: $\sigma_{-k}(P_v)(\N)>0\Leftrightarrow P_v(\Z_{\geq k})>0\LR k\leq  \ell_R(v)+1$, and
 $$\max(\supp(P^L_v))\leq \max(\supp(P_v)).$$
Thus, put $k:=\ell_R(v)+2$, for which we have $P_v(\Z_{\geq k})=0$, $P^L_v(\Z_{\geq k})=0$, and hence
$$\mu_v(\N)=\lim\limits_{n\to\infty}P_{vRL^n}(\N)=P_{vRL^k}(\N).$$
By the reverse property, and by the lemma assumption,  $P_{vRL^k}(\N)=P_{L^kR\overleftarrow{v}}(\N)>\frac12$.
\end{proof}
Coming back to the language of binary digits, by Observation \ref{binarr}:
\begin{conclusion}
If $P_t(\N)>\frac12$ for any odd $t=(t_1t_2\dots t_n)_2$ with $t_2=0$, then $\mu_t(\N)>\frac12$ for any $t\in\N$. 
\end{conclusion}
Our numerical experiments support  \eqref{concon} (and \eqref{medP}). Let $T[3,K]:=\{s\in T | 3\leq s\leq K\}$. For $K=12\ 000\ 001$, we have verified that  the global minima of the function
$$V\colon  T[3,K] \ni t \longmapsto \sum\limits_{d=0}^{\infty}P_t(\{d\}) \in [0,1] $$
satisfy 
$$V(t)=\frac{1}{2} \mbox{ and } W_1(t)=R,$$
where $t=W_1(t)...W_k(t)\in\{L,R\}^\star$. More precisely, the minimum $V(t)=\frac12$ is attained at  $982$ odd integers with $3\leq t\leq12 000 001$ (and the second digit in the binary representation is always "one"). Below we list the first few minimizers. While the symmetries provide some insight,  a general pattern for finding $t$ with $V(t)=\frac12$ is not obvious.

\begin{table}[h]
\centering
\scriptsize
\setlength{\tabcolsep}{3pt}
\begin{tabular}{c c c|c c c|ccc}
\midrule
$\ell(w)$ & $t$ & $w(t)$ & $\ell(w)$ & $t$ & $w(t)$ & $\ell(w)$ & $t$ & $w(t)$ \\
\midrule

0 & 3   & $\varepsilon$ & 7 & 447  & RLRRRRR   & 9 & 1919 & RRLRRRRRR  \\
1 & 7   & R             & 7 & 479  & RRLRRRR   & 9 & 1975 & RRRLRRLRR  \\
2 & 15  & RR            & 7 & 495  & RRRLRRR   & 9 & 1983 & RRRLRRRRR \\
3 & 27  & RLR           & 7 & 503  & RRRRLRR   & 9 & 2015 & RRRRLRRRR \\
3 & 31  & RRR           & 7 & 507  & RRRRRLR   & 9 &  2031 & RRRRRLRRR \\
4 & 55  & RLRR          & 7 & 511  & RRRRRRR   & 9 & 2039 & RRRRRRLRR\\
4 & 59  & RRLR          & 8 & 895  & RLRRRRRR  & 9 & 2043 & RRRRRRRLR \\
4 & 63  & RRRR          & 8 & 951  & RRLRRLRR  & 9 &  2047 & RRRRRRRRR \\
5 & 111 & RLRRR         & 8 & 959  & RRLRRRRR  & 9 & 2039 & RRRRRRLRR \\
5 & 119 & RRLRR         & 8 & 991  & RRRLRRRR  & 9 & 2043 & RRRRRRRLR \\
5 & 123 & RRRLR         & 8 & 1007 & RRRRLRRR  & 9 & 2047 & RRRRRRRRR \\
5 & 127 & RRRRR         & 8 & 1015 & RRRRRLRR  &   10 & 3583 & RLRRRRRRRR       \\
6 & 223 & RLRRRR        & 8 & 1019 & RRRRRRLR  &   10 & 3807 & RRLRRLRRRR         \\
6 & 239 & RRLRRR        & 8 & 1023 & RRRRRRRR  &  10 & 3823 & RRLRRRLRRR           \\
6 & 247 & RRRLRR        &    9 & 1791 & RLRRRRRRR          &  10 & 3831 & RRLRRRRLRR           \\
6 & 251 & RRRRLR        &     9 & 1903 & RRLRRLRRR          &   10 & 3839 & RRLRRRRRRR           \\
6 & 255 & RRRRRR        &     9 & 1911 & RRLRRRLRR      &   10 & 3951 & RRRLRRLRRR           \\

\bottomrule
\end{tabular}
\caption{The first few words with $V(t)=P_t(\N)=\frac12$  }
\end{table}

\end{document}